\documentclass[12pt]{article}
\usepackage{amssymb,amsmath, amsthm}
\usepackage{color,xcolor}
\newcommand{\eb}{\begin{equation}}
\newcommand{\ee}{\end{equation}}
\newcommand{\ebx}{\begin{equation*}}
\newcommand{\eex}{\end{equation*}}
\newtheorem{lemma}{Lemma}[section]
\newtheorem{proposition}[lemma]{Proposition}
\newtheorem{theorem}[lemma]{Theorem}

\newtheorem{corollary}[lemma]{Corollary}
\newtheorem{definition}[lemma]{Definition}
\newtheorem{remark}[lemma]{Remark}
\newtheorem{example}[lemma]{Example}

\allowdisplaybreaks

\makeatletter
\renewcommand*\env@matrix[1][*\c@MaxMatrixCols c]{%
  \hskip -\arraycolsep
  \let\@ifnextchar\new@ifnextchar
  \array{#1}}
\makeatother

\begin{document}

\title{Gap phenomenon of holomorphic maps between Shilov boundaries of type-I bounded symmetric domains}
\author{Yun Gao\thanks{School of Mathematical Sciences, Shanghai Jiao Tong University, Shanghai,
		People's Republic of China. \textbf{Email:}~gaoyunmath@sjtu.edu.cn.
		Supported by NFSC, No.12471042 and No.12471078}}

\maketitle

\begin{abstract}
Motivated by the gap phenomenon for proper holomorphic maps between complex unit balls, this paper investigates smooth CR maps $f$ from an open piece $M$ of the Shilov boundary of the unit ball $B^s$ into the Shilov boundary $S_{r',s'}$ of a higher-rank Type I bounded symmetric domain $\Omega_{r',s'}$. To the best of our knowledge, this is the first work to systematically establish a general gap phenomenon in the higher-rank setting. Our main result demonstrates that when the signature difference $s'-r'$ falls into the interval $k(s-1) \le s'-r' < (k+1)(s-1)$ for some integer $k$, the non-trivial component of $f$ is constrained to a much smaller Type I boundary. Up to automorphisms of the domain and target spaces, $f$ decomposes into the block-diagonal form
$$f(z) = \begin{pmatrix} I_{r'-k} & 0 \\ 0 & \phi(z) \end{pmatrix},$$
where the essential component $\phi: M \to S_{k, s'-r'+k}$ is a smooth CR map into the corresponding Shilov boundary. We provide explicit constructions showing that these dimensional bounds are sharp. As an immediate corollary, in the initial gap regime $s-1 \le s' - r' < 2s-2$, the map $f$ reduces to the standard linear embedding.
\end{abstract}


\section{Introduction}

The rigidity of holomorphic and Cauchy-Riemann (CR) maps between bounded symmetric domains is a classical topic in several complex variables. While the rank 1 case (complex unit balls) is deeply understood from Poincaré to the seminal works of Cima-Suffridge \cite{CS},D'Angelo \cite{Da1, Da2}, Faran \cite{Fa}, Forstnerič \cite{F1, F2}, Huang \cite{Hu1, Hu2}, Huang-Ji \cite{HJ} and others, rigidity phenomena in higher-rank and mixed-rank settings remain largely open.

A celebrated feature of the rank $1$ setting is the \textit{gap phenomenon}. For instance, Faran \cite{Fa} proved that for target dimensions $N \in [n+1, 2n-2]$, any rational proper map from ${B}^n$ to ${B}^N$ is equivalent to the standard linear embedding. Subsequent works by Huang-Ji-Xu \cite{hjx} and Huang-Ji-Yin \cite{HJY} discovered that similar gap phenomena recur over higher intervals, culminating in the general Gap Conjecture \cite{HJY2}.

In contrast to the rank $1$ setting, mappings into higher-rank bounded symmetric domains pose significant challenges. When both domains have rank at least $2$ (i.e., $2 \le \text{rank}(\Omega) \le \text{rank}(\Omega')$), substantial structural rigidity has been established. For comprehensive results in this direction, we refer the reader to the works of Chan \cite{Ch1, Ch2}, Henkin-Novikov \cite{HN}, Kim \cite{Kim1}, Kim-Mok-Seo \cite{KMS}, Kim-Zaitsev \cite{KZ13, KZ15}, Mok \cite{Mok}, Mok-Tsai \cite{MT}, Mok-Ng-Tu \cite{MNT}, Ng \cite{Ng1, Ng2, Ng3}, Seo \cite{Seo1, Seo2, Seo3}, and Tu \cite{Tu1, Tu2}, to name a few.

From a broader geometric perspective, it is natural to expect that systematic gap phenomena are not exclusive to rank $1$ domains, but should universally exist for mappings between bounded symmetric domains of arbitrary ranks. As a first step toward uncovering such a universal gap theory, it is strategically crucial to understand the transition from rank $1$ to higher-rank spaces. However, this mixed-rank case, specifically, mappings from rank $1$ complex balls into higher-rank Type I domains resist classical approaches. Traditional tools like the Tanaka-Chern-Moser theory \cite{CM, Ta} and the Cartan moving frame method do not easily adapt to this mixed-rank setting.

Recently, important progress has been made by Kim and Zaitsev \cite{KZ13}, and Kim \cite{Kim2} in studying the rigidities of local CR mappings between Shilov boundaries of Type I bounded symmetric domains. Notably, Kim \cite{Kim2} studied nonconstant smooth CR maps from the sphere $S_{1,s}$ to $S_{r',s'}$, revealing a specific initial gap interval $[2s-2, 3s-4]$. Despite this breakthrough, whether a systematic, multi-interval gap phenomenon analogous to the celebrated Gap Conjecture in the rank $1$ setting exists for higher-rank bounded symmetric domains has remained unknown. In this paper, we provide the first systematic investigation of this problem and establish the general, sharp gap intervals for CR maps from rank $1$ Type I Shilov boundaries to higher-rank Type I Shilov boundaries.

To overcome the analytical difficulties of the mixed-rank setting, we introduce a geometric framework. Let $r, s \in \mathbb{N}$ with $r \le s$, and let $\mathbb{C}^{r,s}$ denote the complex vector space $\mathbb{C}^{r+s}$ equipped with the indefinite inner product $\langle \cdot, \cdot \rangle_{r,s}$ of signature $(r,s)$. A linear subspace in $\mathbb{C}^{r,s}$ is called positive (resp. null) if $\langle\cdot,\cdot\rangle_{r,s}$ restricts to a positive-definite (resp. null) form on it. The Type I domain $\Omega_{r,s}$ parametrizes such positive $r$-planes, while its Shilov boundary $S_{r,s} := S(\Omega_{r,s})$ naturally corresponds to the null $r$-planes. 

Instead of relying on PDE computations on the boundary manifolds, we consider the ambient Grassmannian variety $G(r,r+s)$ equipped with an orthogonal structure induced by $\langle\cdot,\cdot \rangle_{r,s}$. Specifically, two points $z, w \in G(r,r+s)$ are defined to be orthogonal (denoted by $z \perp w$) if their corresponding subspaces are orthogonal in $\mathbb{C}^{r,s}$. We denote this orthogonal Grassmannian by $\mathcal{G}(r,s)$. 

Within this geometric framework, our primary approach is to investigate local holomorphic maps $F: U \to \mathcal{G}(r',s')$ where $U \subseteq \mathcal{G}(r,s)$ is a connected open set that preserve this orthogonal structure that is, $F(z) \perp F(w)$ for any $z, w \in U$ satisfying $z \perp w$. We refer to such mappings as \emph{orthogonal maps}. Crucially, when the domain $U$ contains a null point, these local orthogonal maps between Grassmannians correspond precisely to local smooth CR mappings between the Shilov boundaries via standard holomorphic extension (as we will detail in Proposition \ref{orthShilov}). 

By leveraging this purely algebraic and geometric correspondence, we establish the following general sharp gap intervals for Type I domains:

\begin{theorem}\label{main_shilov}
	Let $s \ge 3$, and let $r', s' \in \mathbb{N}$ and $k \in \mathbb{Z}_{\ge 0}$ with $2 \le r' \le s'$ and $k < r'$. Let $f$ be a smooth CR mapping between connected open pieces of the Shilov boundaries of two Type I bounded symmetric domains $\Omega_{1,s}$ and $\Omega_{r',s'}$. If the signatures satisfy the gap condition
	$$k(s-1) \le s' - r' < (k+1)(s-1),$$
	then after composing with suitable automorphisms of $\Omega_{1,s}$ and $\Omega_{r',s'}$,  $f$ is given by:
	$$f(z) = \begin{pmatrix} I_{r'-k} & 0 \\ 0 & \phi(z) \end{pmatrix}$$
	where $\phi$ is a smooth CR map between the corresponding Shilov boundaries of two Type I bounded symmetric domains $\Omega_{1,s}$ and $\Omega_{k,s'-r'+k}$

\end{theorem}

As a direct consequence of Theorem \ref{main_shilov}, we obtain the following non-existence result for special proper maps:

\begin{corollary}\label{cor:non_existence}
	Let $s \ge 3$, and let $r', s', k \in \mathbb{N}$ with $2 \le r' \le s'$ and $k < r'$. If the signatures satisfy 
	$$k(s-1) \le s' - r' < (k+1)(s-1),$$
	then there does not exist any proper holomorphic map from $B^s$ to $\Omega_{r',s'}$ that extends smoothly up to the boundary and preserves the Shilov boundaries.
\end{corollary}

Furthermore, as a direct consequence of the sharp bound when $k=1$, we obtain the following strong rigidity corollary:

\begin{corollary}\label{main2}
	Let $s \ge 2$, and let $r', s' \in \mathbb{N}$ with $2 \le r' \le s'$. Let $f$ be a smooth CR mapping between connected open pieces of the Shilov boundaries of two Type I bounded symmetric domains $\Omega_{1,s}$ and $\Omega_{r',s'}$.
	\begin{itemize}
		\item[(1)] If $s' - r' < s - 1$, then $f$ is a constant map.
		
		\item[(2)] If $s - 1 \le s' - r' < 2s - 2$, then after composing with suitable automorphisms of $\Omega_{1,s}$ and $\Omega_{r',s'}$, $f$ is given by the standard linear embedding:
		$$ z \mapsto \begin{pmatrix}
			I_{r'-1} & 0 & 0 \\
			0 & z & 0 
		\end{pmatrix}. $$
	\end{itemize}
\end{corollary}

It is worth noting that our proof of Theorem \ref{main_shilov} centers entirely on the orthogonality preservation property on the ambient Grassmannians and the dimensional capacity constraints deduced from it. This geometric perspective allows us to bypass the PDE computations relied upon in previous literature, offering a direct and conceptual approach to the problem.

\section{Orthogonal structure on Grassmannian}

In \cite{GN}, the authors introduced an orthogonal structure in projective space to solve rigidity problems for proper holomorphic mappings among generalized balls. Building upon this framework, this section generalizes the concept of orthogonal structures to Grassmannian manifolds.

Let $r,s,t \in \mathbb{N}$ and $n := r+s+t > 0$. We define the indefinite inner product of signature $(r;s;t)$ on $\mathbb{C}^{n}$ by:
$$
	\langle z, w\rangle_{r,s,t}
	= z_1\bar w_1 + \cdots + z_r\bar w_r - z_{r+1}\bar w_{r+1} - \cdots - z_{r+s}\bar w_{r+s},
$$
where $z=(z_1,\ldots,z_{n})$ and $w=(w_1,\ldots,w_{n})$. We also define the indefinite squared norm $\|z\|^2_{r,s,t} = \langle z, z\rangle_{r,s,t}$. A vector $z \in \mathbb{C}^{r,s,t}$ endowed with this inner product is called a \textit{positive vector} if $\|z\|^2_{r,s,t} > 0$, a \textit{negative vector} if $\|z\|^2_{r,s,t} < 0$, and a \textit{null vector} if $\|z\|^2_{r,s,t} = 0$. Two vectors $z, w$ are called orthogonal (denoted $z \perp w$) if $\langle z, w\rangle_{r,s,t} = 0$. The \textit{orthogonal complement} of $z$ is defined as
$$z^{\perp} = \{w \in \mathbb{C}^{r,s,t} \mid \langle z, w\rangle_{r,s,t} = 0\}.$$

\begin{remark}\label{rem:segre_connection}
	In the terminology of several complex variables, the orthogonal complement $z^{\perp}$ corresponds  to the \textit{Segre variety} associated with the point $z$. Specifically, for the real-analytic null cone defined by $\langle z, z \rangle_{r,s,t} = 0$, complexifying the Hermitian conjugate yields the Segre hypersurface $Q_z := \{w \in \mathbb{P}^{r,s,t} \mid \langle w, z \rangle_{r,s,t} = 0\} = z^\perp$. Under this correspondence, our orthogonality preservation condition $F(z) \perp F(w)$ is geometrically equivalent to the invariance of Segre varieties under the mapping $F$.
\end{remark}

We denote by $\mathbb{C}^{r,s,t}$ the space $\mathbb{C}^{n}$ equipped with the Hermitian inner product defined above, and we set $\mathbb{P}^{r,s,t} := \mathbb{P}\mathbb{C}^{r,s,t}$. For brevity, we write $\mathbb{C}^{r,s}$ and $\mathbb{P}^{r,s}$ instead of $\mathbb{C}^{r,s,0}$ and $\mathbb{P}^{r,s,0}$.

Consider a complex linear subspace $V \subset \mathbb{C}^{r,s,t}$ such that the restriction of $\langle\cdot,\cdot\rangle_{r,s,t}$ to $V$ has signature $(a; b; c)$; we refer to $V$ as an \textit{$(a,b,c)$-subspace} of $\mathbb{C}^{r,s,t}$. 

If $a = b = 0$, $V$ is called a \textit{null space}. Similarly, it is called a \textit{positive space} (resp. \textit{negative space}) if $b = c = 0$ (resp. $a = c = 0$). We will also use the terms \textit{null $k$-plane}, \textit{positive $k$-plane}, and \textit{negative $k$-plane} when $\dim V = k$. Obviously, the maximum dimension of a null space in $\mathbb{P}^{r,s}$ is $\min\{r,s\}-1$. 

Let $r, s$ be positive integers with $r \le s$. Consider the Grassmannian variety $G(r,r+s)$, whose closed points are in bijection with $r$-dimensional vector subspaces $V \subset \mathbb{C}^{r+s}$. Fixing the standard basis $(e_1, \ldots, e_{r+s})$ of $\mathbb{C}^{r+s}$, a linear subspace $V$ can be represented as the row space of an $r \times (r+s)$ matrix:
$$A = \left(\begin{array}{cccc}
	a_{11} & a_{12} & \cdots & a_{1,r+s} \\
    a_{21} & a_{22} & \cdots & a_{2,r+s} \\
	\vdots & \vdots & \ddots & \vdots \\
	a_{r1} & a_{r2} & \cdots & a_{r,r+s} \\
\end{array}\right),
$$
where the rows of $A$ encode a set of spanning vectors of $V$ with respect to the standard basis. Such a matrix $A$ is called a \textit{representative matrix} of $V$.

The matrix $A$ is not unique: left multiplication by any invertible $r \times r$ matrix $X$ yields another representative matrix $XA$ for the same subspace $V$. Conversely, if two matrices $A$ and $A'$ represent the same subspace, there exists an invertible $r \times r$ matrix $X$ such that $A' = XA$.

\textbf{Notation:} For a point $\mathcal Z \in G(r,r+s)$, let $V_{\mathcal Z} \subset \mathbb{C}^{r+s}$ denote the $r$-dimensional linear subspace corresponding to $\mathcal Z$, and let $A_{\mathcal Z}$ denote a representative matrix of $V_{\mathcal Z}$. If the row vectors $\alpha_1, \ldots, \alpha_r$ of $A_{\mathcal Z}$ are pairwise orthogonal with respect to the indefinite Hermitian inner product on $\mathbb{C}^{r,s}$ (i.e., $\alpha_i \perp \alpha_j$ in $\mathbb{C}^{r,s}$), then $A_{\mathcal Z} = (\alpha_1, \ldots, \alpha_r)^t$ is called an \textit{orthogonal representative matrix} for ${\mathcal Z}$. By equipping $\mathbb{C}^{r,s}$ with an indefinite inner product, we naturally induce an orthogonal structure on the Grassmannian $G(r,r+s)$. This structure categorizes subspaces based on their metric properties:

\begin{definition}
Let two points $\mathcal Z, \mathcal W \in G(r,r+s)$ be represented by $r \times (r+s)$ matrices $A_{\mathcal Z}$ and $A_{\mathcal W}$. Using the signature matrix 
$$ I_{r,s} = \begin{pmatrix} I_r & 0 \\ 0 & -I_s \end{pmatrix}, $$ 
we say ${\mathcal Z}$ is orthogonal to $\mathcal W$ (denoted $\mathcal Z \perp \mathcal W$) if and only if
\[
A_{\mathcal Z} I_{r,s} A_{\mathcal W}^H = 0,
\]
where $A^H = \bar{A}^t$ denotes the Hermitian transpose.

A point $\mathcal{Z} \in \mathcal{G}(r,s)$ is classified as:
\begin{itemize}
	\item[(1)] \emph{null} if $A_{\mathcal{Z}} I_{r,s} A_{\mathcal{Z}}^H = 0$,
	\item[(2)] \emph{positive} if $A_{\mathcal{Z}} I_{r,s} A_{\mathcal{Z}}^H > 0$,
	\item[(3)] \emph{semi-positive} if $A_{\mathcal{Z}} I_{r,s} A_{\mathcal{Z}}^H \ge 0$.
\end{itemize}
Here, $\mathcal{G}(r,s)$ denotes the Grassmannian $G(r,r+s)$ endowed with this orthogonal structure.
\end{definition}

It is clear that for any two points ${\mathcal Z}, {\mathcal W} \in \mathcal{G}(r,s)$,
$${\mathcal Z} \perp \mathcal W \iff V_{\mathcal Z} \subseteq (V_{\mathcal W})^\perp \subseteq \mathbb{C}^{r,s},$$ 
where $(V_{\mathcal W})^\perp$ denotes the orthogonal complement with respect to $\langle \cdot, \cdot \rangle_{r,s}$.

Hence, if $\mathcal Z \perp \mathcal W$ in $\mathcal{G}(r,s)$, then every vector $\alpha \in V_{\mathcal Z}$ is orthogonal to every vector $\beta \in V_{\mathcal W}$ with respect to $\langle \cdot, \cdot \rangle_{r,s}$. Furthermore, we observe that $\mathcal{G}(1,s)$ coincides with the projective space $\mathbb{P}^{1,s}$. 

Let $M_{r,s}$ denote the space of $r \times s$ complex matrices. The \emph{Type I irreducible bounded symmetric domain} $\Omega_{r,s} \subseteq M_{r,s} \cong \mathbb{C}^{r \times s}$ is defined by
\[
\Omega_{r,s} = \{ Z \in M_{r,s} \mid I_r - ZZ^H > 0 \}.
\]
Here, $> 0$ denotes the positive definiteness of a square matrix. This domain corresponds bijectively to the set of positive $r$-planes in $\mathbb{C}^{r,s}$, making $\Omega_{r,s}$ precisely the set of positive points in $\mathcal{G}(r,s)$. The \emph{Shilov boundary} of $\Omega_{r,s}$, given by
\[
S_{r,s} = \{ Z \in M_{r,s} \mid I_r = ZZ^H \},
\]
coincides with the set of null points in $\mathcal{G}(r,s)$. Moreover, if $Z \in S_{r,s}$, then $V_{\mathcal{Z}}$ is a maximal null subspace of $\mathbb{C}^{r,s}$ where $\mathcal{Z}=[I_r,Z]$.

\begin{definition}\label{orth}
Let $U \subseteq \mathcal{G}(r,s)$ be a connected open set containing a null point. A holomorphic map $F: U \to \mathcal{G}(r',s')$ is called \emph{orthogonal} if $F(\mathcal Z) \perp F(\mathcal W)$ for all $\mathcal Z, \mathcal W \in U$ such that $\mathcal Z \perp \mathcal W$. We also refer to $F$ as a \emph{local orthogonal map} from $\mathcal{G}(r,s)$ to $\mathcal{G}(r',s')$.
\end{definition}

From the definition of orthogonal maps and the properties of the Shilov boundary, we immediately derive the following result.

\begin{proposition}\label{orthShilov}
	Let $r,s > 0$ and $F: U \subset \mathcal{G}(r,s) \to \mathcal{G}(r',s')$ be a holomorphic map, where $U$ is a connected open set containing a null point. If $F$ maps null points to null points, there exists an open set $U' \subset U$ such that $F|_{U'}: U' \to \mathcal{G}(r',s')$ is an orthogonal map.
\end{proposition}

\begin{proof}
	By composing with automorphisms and shrinking $U$ if necessary, we may assume without loss of generality that $U$ is contained in the affine chart $U_1 \subset \mathcal{G}(r,s)$, where $U_1 = \{\mathcal{Z} = [I_r, Z] \mid Z = (z_{ij})_{r \times s} \in \mathbb{C}^{r \times s}\}$, and that $F(U)$ lies in the analogous chart $U_1' = \{\mathcal{Z}' = [I_{r'}, Z'] \mid Z' \in \mathbb{C}^{r' \times s'}\}$. 
	
	In these coordinates, null points $\mathcal{Z}$ in $\mathcal{G}(r,s)$ are characterized by the equation $I_r - ZZ^H = 0$. Explicitly, $\mathcal{Z} = [I_r, Z] \in \mathcal{G}(r,s)$ is a null point if and only if 
	$$h_{ij}(Z, \bar{Z}) := \sum_{k=1}^{s} z_{ik}\bar{z}_{jk} - \delta_{ij} = 0, \quad (1 \le i \le r, \quad 1 \le j \le r),$$ 
	where $\delta_{ij}$ is the Kronecker delta. 
	
	Assume that $F(\mathcal{Z}) = [I_{r'}, F^{\sharp}(Z)]$ with $F^{\sharp} = (f_{ij})_{r' \times s'}$. Define
	$$G_{ij}(Z, \bar{Z}) := \sum_{k=1}^{s'} f_{ik}(Z)\overline{f_{jk}(Z)} - \delta_{ij}, \quad (1 \le i \le r', \quad 1 \le j \le r'),$$
	where $f(Z)$ denotes $f(z_{11}, \ldots, z_{rs})$.
	
	Since $F$ maps null points to null points, there exists a connected open set $U' \subset U$ containing a null point and real analytic functions $g_{ij}^{kl}(Z, \bar{Z})$ such that
	$$G_{ij}(Z, \bar{Z}) = \sum_{k=1}^{s'} f_{ik}(Z)\overline{f_{jk}(Z)} - \delta_{ij} = \sum_{\mu, \nu=1}^{r} h_{\mu\nu}(Z, \bar{Z})g_{ij}^{\mu\nu}(Z, \bar{Z})$$ 
	holds on $U'$. 
	
	By polarization (after further shrinking $U'$ if necessary), this implies that for all $\mathcal{Z} = [I_r, Z]$ and $\mathcal{W} = [I_r, W]$ in $U'$, we have 
	\begin{equation}\label{bound eq}
		G_{ij}(Z, \bar{W}) = \sum_{k=1}^{s'} f_{ik}(Z)\overline{f_{jk}(W)} - \delta_{ij} = \sum_{k,l} h_{kl}(Z, \bar{W})g_{ij}^{kl}(Z, \bar{W}).
	\end{equation}
	
	For any point $\mathcal{Z} = [I_r, Z] \in \mathcal{G}(r,s)$, a point $\mathcal{W} = [I_r, W]$ lies in $\mathcal{Z}^{\perp} \cap U'$ precisely when $h_{\mu\nu}(Z, \bar{W}) = 0$ for all $k, l$. By \eqref{bound eq}, this forces $G_{ij}(Z, \bar{W}) = 0$, which implies $F(\mathcal{W}) \in F(\mathcal{Z})^{\perp}$. Hence, $F$ preserves orthogonality locally, proving that $F|_{U'}$ is an orthogonal map. 
\end{proof}

\section{Properties of local orthogonal maps between Grassmannians}

First, we introduce the relationship between smooth CR mappings and orthogonal mappings.

\begin{theorem} \label{thm:cr_extension_local}
	Let $s' \ge s \ge 2$, $s'\ge r'$ and let $\Sigma \subset S_{1,s}$ be a connected open piece of the Shilov boundary. Suppose $f: \Sigma \to S_{r',s'}$ is a smooth CR mapping. Then $f$ admits a unique holomorphic extension to a global rational mapping $F$. Furthermore, $F$ is holomorphic on an open neighborhood $U \subset \mathbb{C}^s$ containing the closed ball $\bar{B}^s$, and its restriction to the interior satisfies $F(B^s) \subset \bar{\Omega}_{r',s'}$.
\end{theorem}

\begin{proof}

Since $\Sigma$ lies in a smooth strictly pseudoconvex CR manifold, classical local extension theorems in \cite{Tum} guarantee that the smooth CR mapping $f$ extends holomorphically to a wedge domain attached to $\Sigma$. Denote this local matrix-valued holomorphic extension by $F$.

 Write $H_i(z)$ for the $i$-th row of $F(z)$, $1 \le i \le r'$. As $f$ maps $\Sigma$ into the Shilov boundary $S_{r',s'}$, the identity $F(z)F(z)^H = I_{r'}$ holds on $\Sigma$. Hence
$$
H_i(z) H_i(z)^H = 1 \quad \text{for all } z\in\Sigma.
$$
Each component of $f$ is a smooth CR function on $\Sigma$. Identify the Shilov boundary $S_{1,s}$ with the unit sphere $S^{2s-1} \subset \mathbb{C}^{s}$. Therefore,  $H_i|_\Sigma\colon \Sigma \subset S^{2s-1} \to S^{2s'-1}$, defines a local smooth CR map from $\Sigma$ into the unit sphere in $\mathbb{C}^{s'}$. Theorem 1.4 in Forstnerič \cite{F2} asserts that any smooth local CR mapping between such spheres extends to a global rational mapping. In addition, the theorem further implies that this extended rational map is holomorphic throughout the interior ball $B^s$, sends $B^s$ into $B^{s'}$, and has no poles on the boundary $\partial B^s$. 

Since every row $H_i$ extends to a rational map, the full matrix-valued mapping $F = (H_1,\dots,H_{r'})^T$ yields a rational map $F\colon \mathbb{C}^s \dashrightarrow M_{r',s'}$. Write each extended rational row as $H_i = P_i/Q_i$. The denominator $Q_i$ is nonvanishing on the compact sphere $\partial B^s$, hence remains nonvanishing on some small neighborhood of $\partial B^s$. Taking the finite intersection of these neighborhoods over all rows produces an open neighborhood $U\subset\mathbb{C}^s$ containing $\bar{B}^s$, on which $F$ is holomorphic.

As $F$ is holomorphic on $U$, the matrix-valued function
\[
P(z,\bar{z}) := F(z)F(z)^H - I_{r'}
\]
is real-analytic on $U$ and $P(z,\bar{z})=0$ for all $z\in\Sigma$. By the identity principle for real-analytic functions on connected real-analytic manifolds, vanishing on $\Sigma$ forces $P(z,\bar{z})\equiv 0$ on the whole sphere. We obtain the global boundary identity:
\[
F(z)F(z)^H = I_{r'} \quad \text{for all } z\in\partial B^s.
\]
So the operator norm satisfies $\|F(z)\|_{\text{op}} = 1$ for all $z\in\partial B^s$. Fix arbitrary unit vectors $u_0\in\mathbb{C}^{r'}$ and $v_0\in\mathbb{C}^{s'}$. Define
$
g_{u_0,v_0}(z) := u_0^H F(z) v_0.
$
On $\partial B^s$, the Cauchy–Schwarz inequality yields
\[
|g_{u_0,v_0}(z)| = |u_0^H F(z) v_0| \le \|u_0\|_2 \cdot \|F(z)v_0\|_2 \le \|F(z)\|_{\text{op}} = 1.
\]
By the classical maximum modulus principle, $|g_{u_0,v_0}(z)|=|u_0^H F(z) v_0| \le 1$ for all $z \in B^s$.

Taking the supremum over all unit vectors $u_0, v_0$, we conclude
\[
\|F(z)\|_{\text{op}} = \sup_{\|u\|_2=1,\,\|v\|_2=1} \big|u^H F(z) v\big| \le 1 \quad \text{for all } z\in B^s.
\]
This is algebraically equivalent to $I_{r'} - F(z)F(z)^H \ge 0$ for all $z\in B^s$, which exactly states $F(B^s) \subset \bar{\Omega}_{r',s'}$.

Finally, uniqueness of the holomorphic extension follows directly from the identity principle. If $F_1$ and $F_2$ are two such holomorphic extensions, their difference $F_1-F_2$ vanishes on the connected open set $\Sigma$. Since $\Sigma$ is a nonempty open subset of a generic real-analytic CR manifold, the identity principle for holomorphic functions on generic submanifolds forces $F_1 \equiv F_2$ identically on $U$. The proof is complete.
\end{proof}

Combining with Proposition \ref{orthShilov}, any smooth CR mapping $f:\Sigma\subset S_{1,s} \to S_{r',s'}$ extends to a local orthogonal map $F$ from $\mathbb{P}^{1,s}$ to $\mathcal{G}(r',s')$ which maps positive points to semi-positive points.

Now, we will introduce some properties of orthogonal maps between Grassmannians. The following proposition is straightforward.

\begin{proposition}\label{intersection}
	If two points $\mathcal{Z}, \mathcal{W} \in \mathcal{G}(r,s)$ satisfy $\mathcal{Z} \perp \mathcal{W}$ and $\dim(V_{\mathcal{Z}} \cap V_{\mathcal{W}}) > 0$, then every vector $\alpha \in V_{\mathcal{Z}} \cap V_{\mathcal{W}}$ is a null vector in $\mathbb{C}^{r,s}$.
\end{proposition}

\textbf{Notation:} For any subset $H \subset \mathcal{G}(r,s)$, we use $V_{H}$ to denote the linear span of $\cup_{z \in H} V_{z}$ in $\mathbb{C}^{r,s}$.

\begin{lemma} \label{lem:unified_local_nullity}
	Let $s \ge 3$, and let $r', s', k \in \mathbb{N}$ with $2 \le r' \le s'$. Let $F: U \to \mathcal{G}(r',s')$ be a local orthogonal map, where $U \subset \mathcal{G}(1,s)$ is a connected open set containing a null point. If $k(s-1) \le s' - r' < (k+1)(s-1)$ and $r' \ge k$, then there exists a dense analytic open subset $U^\circ \subset U$ and an integer $m \ge r'-k$ such that the dimension of the null space $N_z := V_{F(z)} \cap (V_{F(z)})^\perp$ is exactly $m$ for all $z \in U^\circ$, and $\dim N_z > m$ for all $z \in U \setminus U^\circ$.
\end{lemma}

\begin{proof}
	Locally on $U$, we can lift the map $F$ to a holomorphic matrix-valued function $A(z)$ of size $r' \times (r'+s')$, where the rows of $A(z)$ span the $r'$-dimensional subspace $V_{F(z)}$. The restriction of the target indefinite Hermitian form to $V_{F(z)}$ is given by the $r' \times r'$ Gram matrix:
	$$ M(z) = A(z)I_{r',s'}A(z)^H. $$
	Let $R_{max}$ be the maximum rank of $M(z)$ achieved on $U$. The set where this maximum rank is attained, $U^\circ = \{z \in U \mid \mathrm{rank}\, M(z) = R_{max}\}$, is a dense analytic open subset of $U$. The dimension of the null space is precisely $\dim N_z = r' - \mathrm{rank}\, M(z)$. Therefore, on $U^\circ$, this dimension is a constant $m := r' - R_{max}$, and for any $z \in U \setminus U^\circ$, the rank strictly drops, yielding $\dim N_z > m$.
	
	We now prove that this generic minimal nullity satisfies $m \ge r'-k$. Let $p \in U$ be a null point. Then $p \in p^\perp$. We will iteratively construct $s+1$ mutually orthogonal non-null points $z_1, \ldots, z_{s+1}$ entirely within the generic locus $U^\circ$.
	
	Since $U^\circ$ is a dense analytic open set in $U$, we can choose a positive point $z_1 \in U^\circ$ arbitrarily close to $p$. The orthogonal complement $z_1^\perp$ intersects $U$ in a non-empty open subset $U_1 = z_1^\perp \cap U$. Then $U^\circ$ is also dense in $U_1$. Thus, the intersection $U_1 \cap U^\circ$ is non-empty and contains points arbitrarily close to $p$, allowing us to choose $z_2 \in U_1 \cap U^\circ$ near $p$. 
	Proceeding inductively, suppose mutually orthogonal points $z_1, \dots, z_l \in U^\circ$ have been chosen near $p$. Their orthogonal complement $(z_1 \oplus \dots \oplus z_l)^\perp$ intersects $U$ in a non-empty open subset $U_l$. The denseness of $U^\circ$ again ensures that $U_l \cap U^\circ \neq \emptyset$, allowing us to choose the next point $z_{l+1} \in U_l \cap U^\circ$ near $p$. Repeating this process yields $s+1$ mutually orthogonal non-null points $z_1, \dots, z_{s+1}$ entirely within $U^\circ$.
	
	Because all $z_i \in U^\circ$, we strictly have $\dim N_{z_i} = m$ for all $1 \le i \le s+1$. For $i \neq j$, the orthogonality $z_i \perp z_j$ implies $F(z_i) \perp F(z_j)$, yielding $V_{F(z_j)} \subset (V_{F(z_i)})^\perp$. Consequently, for each $i$, we have $\sum_{j \neq i} V_{F(z_j)} \subset (V_{F(z_i)})^\perp$. Intersecting both sides with $V_{F(z_i)}$ yields:
	$$ V_{F(z_i)} \cap \sum_{j \neq i} V_{F(z_j)} \subset V_{F(z_i)} \cap (V_{F(z_i)})^\perp = N_{z_i}. $$
	This provides the precise dimension bound:
	$$ \dim\left(V_{F(z_i)} \cap \sum_{j \neq i} V_{F(z_j)}\right) \le \dim N_{z_i} = m. $$
	Applying the general dimension inequality for sums of vector spaces:
	\begin{align*}
		\dim\left(\sum_{i=1}^{s+1} V_{F(z_i)}\right) &\ge \sum_{i=1}^{s+1} \dim V_{F(z_i)} - \sum_{i=2}^{s+1} \dim\left(V_{F(z_i)} \cap \sum_{j=1}^{i-1} V_{F(z_j)}\right) \\
		&\ge (s+1)r' - s m.
	\end{align*}
	On the other hand, since $N_{z_1}$ is an $m$-dimensional null subspace, its orthogonal complement $(N_{z_1})^\perp$ is an $(r'-m, s'-m, m)$-subspace. For $j \ge 2$, $V_{F(z_j)} \subset (V_{F(z_1)})^\perp \subset (N_{z_1})^\perp$. Moreover, $V_{F(z_1)} \subset (N_{z_1})^\perp$. Therefore,
	$$ \sum_{i=1}^{s+1} V_{F(z_i)} \subset (N_{z_1})^\perp, $$
	which provides the upper bound:
	$$ \dim\left(\sum_{i=1}^{s+1} V_{F(z_i)}\right) \le r' + s' - m. $$
	Combining the lower and upper bounds with the gap condition $s' - r' < (k+1)(s-1)$ and $s \ge 3$, we deduce:
	$$ (s+1)r' - s m \le r' + s' - m \implies (s-1)m \ge s r' - s' > (s-1)r' - (k+1)(s-1). $$
	Dividing by $s-1$ yields $m > r' - (k+1)$. Since $m$ is an integer, we conclude that $m \ge r' - k$. 
\end{proof}

\begin{lemma}\label{fixed_null_space}
	 Let the notations and assumptions be exactly as in Lemma \ref{lem:unified_local_nullity}. If $F$ furthermore maps positive points to semi-positive points, then there exists a fixed $m$-dimensional null subspace $N \subset \mathbb{C}^{r^{\prime},s^{\prime}}$ such that $N_{z}=N$ for all $z \in U^{\circ}$. In particular, $N \subset V_{F(z)}$ for all $z \in U$.
\end{lemma}

\begin{proof}
	By Lemma \ref{lem:unified_local_nullity}, there exists a dense analytic open subset $U^\circ \subset U$ such that the null space $N_z := V_{F(z)} \cap (V_{F(z)})^\perp$ has a minimal constant dimension $m$ for $z \in U^\circ$, and $\dim N_w \ge m$ for all $w \in U$.
	
	For any positive point $z \in U^\circ$,  $V_{F(z)}$ is an $r'$-dimensional semi-positive subspace of a space of signature $(r',s')$,
	$V_{F(z)}^{\perp}$ is negative semi-definite. Consequently, the only null vectors in $(V_{F(z)})^\perp$ are precisely those in its null space $N_z$. If $w \in U$ satisfies $w \perp z$, orthogonality preservation yields $V_{F(w)} \subset (V_{F(z)})^\perp$. Since $N_w \subset V_{F(w)}$ is totally null, we obtain the inclusion $N_w \subset N_z$. The dimension constraint $\dim N_w \ge m = \dim N_z$ strictly forces $N_w = N_z$.
	
	Since positive points are dense near the null point $z_0$, we can choose two positive points $z_1, z_2 \in U^\circ$ arbitrarily close to $z_0$. The orthogonal complement $(V_{z_1} \oplus V_{z_2})^\perp$ has signature $(0, s-1)$.  As $z_1, z_2$ arbitrarily close to $z_0$ and $z_0 \perp z_0$, this projective subvariety passes arbitrarily close to $z_0$, and thus intersects $U$.
	
	Taking any $w \in (V_{z_1} \oplus V_{z_2})^\perp \cap U$, we have $w \perp z_1$ and $w \perp z_2$. By our previous deduction, $N_{z_1} = N_w = N_{z_2}$. This shows that $N_z$ is locally constant on the real open set of positive points in $U^\circ$, and thus freezes into a fixed null subspace $N$. Crucially, while the map $z \mapsto N_z$ is only real-analytic, the target subspace mapping $z \mapsto V_{F(z)}$ is a Grassmannian-valued holomorphic map. By the identity theorem for holomorphic functions, the algebraic inclusion $N \subset V_{F(z)}$ extends from this real open set to the entire connected domain $U$. The proof is complete.
\end{proof}

With the geometric rigidity established in the preceding lemmas, now we can prove our main result, Theorem \ref{main_shilov}. The existence of a globally fixed null subspace essentially ``locks'' a portion of the extended mapping. By applying a suitable target automorphism to normalize this fixed null space, we can decouple the mapping, forcing it into a strict block-diagonal form on the boundary.

\begin{proof}[Proof of Theorem \ref{main_shilov}]Let $\Sigma$ be an open piece in $S_{1,s}$. 
	By Theorem \ref{thm:cr_extension_local}, the local smooth CR mapping $f: \Sigma \to S_{r',s'}$ extends to a local holomorphic map $F: U \subset \mathcal{G}(1,s) \to \mathcal{G}(r',s')$, where $U$ is a connected open neighborhood of $\Sigma$. Since $f$ maps the null points to the null points of $S_{r',s'}$, Proposition \ref{orthShilov} ensures that $F$ is a local orthogonal map. Furthermore, by Theorem \ref{thm:cr_extension_local}, $F$ maps positive points in $U$ to semi-positive points.

	By Lemma \ref{lem:unified_local_nullity} and Lemma \ref{fixed_null_space}, the gap condition $k(s-1) \le s' - r' < (k+1)(s - 1)$ guarantees the existence of a globally fixed $(r'-k)$-dimensional null subspace $N \subset \mathbb{C}^{r',s'}$ such that $N \subset V_{F(z)} \cap (V_{F(z)})^\perp$ for all $z \in U$.
	
	Since $SU(r',s')$ acts transitively on totally isotropic subspaces of a given dimension, we can apply a target automorphism to map $N$ to the standard null subspace. Geometrically, applying this automorphism corresponds to evaluating the map in a suitably adapted affine coordinate chart of the target Grassmannian.
	
	In this standard affine coordinate chart, where a point is represented via its row span $[I_{r'}, Z]$, this standard subspace $N$ is spanned precisely by the first $r'-k$ rows of $[I_{r'}, E]$, where $E = \left(\begin{smallmatrix} I_{r'-k} & 0 \\ 0 & 0 \end{smallmatrix}\right)$.
	
	Because $V_{F(z)}$ contains $N$ and is simultaneously orthogonal to $N$, the transformed subspace $V_{(\Phi\circ F)(z)}$ contains $\Phi(N)$ and is orthogonal to it. We can therefore choose an orthogonal representative matrix for $V_{(\Phi\circ F)(z)}$ adapted to the orthogonal splitting $V_{(\Phi\circ F)(z)} = \Phi(N) \oplus^\perp W_z$. In our chosen affine coordinates, this geometric decoupling trivially forces the coordinate matrix of the transformed map to take the strictly block-diagonal form:
	\[
	Z(V_{(\Phi\circ F)(z)}) = 
	\begin{pmatrix}
		I_{r'-k} & 0 \\
		0 & \tilde{\phi}(z)
	\end{pmatrix}.
	\]
	
	Finally, by applying a suitable domain automorphism $\Psi \in \operatorname{Aut}(\Omega_{1,s})$ to normalize the standard coordinates if necessary, and restricting the extended map back to the Shilov boundary piece $\Psi(\Sigma)$, we obtain the desired form for the original CR map:
	\[
	(\Phi \circ f \circ \Psi^{-1})(z) =
	\begin{pmatrix}
		I_{r'-k} & 0 \\
		0 & \phi(z)
	\end{pmatrix},
	\]
	where $\phi \colon \Psi(\Sigma) \to S_{k,s'-r'+k}$ inherently inherits the property of being a smooth CR map into the corresponding Shilov boundary.
\end{proof}

We construct a family of holomorphic mappings to demonstrate that the signature difference bound $k(s-1) \le s' - r' < (k+1)(s-1)$ in Theorem \ref{main_shilov} is sharp. In other words, for a given gap $s'-r'$ in this range, the constant identity block $I_{r'-k}$ cannot be improved to $I_{r'-k+1}$ in general.

\begin{example}\label{ex:optimal_bound}
	Let $s \ge 3$ and  $k \ge 1$ be an integer satisfying $k(s-1) \le s' - r' < (k+1)(s-1)$. We can write the gap as $s' - r' = k(s-1) + l$, where $l$ is an integer such that $0 \le l < s-1$. For any $z = (z_1, \dots, z_s) \in S_{1,s}$, we define a $k \times ks$ block-diagonal matrix $D_k(z)$ by repeating $z$ along the diagonal:
	$$D_k(z) = \begin{pmatrix} 
		z & 0 & \dots & 0 \\ 
		0 & z & \dots & 0 \\ 
		\vdots & \vdots & \ddots & \vdots \\ 
		0 & 0 & \dots & z 
	\end{pmatrix}.$$
	Now, we define the map $F_k: S_{1,s} \to S_{r', s'}$ by
	$$F_k(z) = \begin{pmatrix} 
		I_{r'-k} & 0 & 0 \\ 
		0 & D_k(z) & 0_{k \times l}
	\end{pmatrix}.$$
\end{example}

\begin{remark}\label{rmk:sharpness}
	The family of maps $F_k$ constructed in Example \ref{ex:optimal_bound} demonstrates that the interval bound $k(s-1) \le s' - r' < (k+1)(s-1)$ for obtaining the identity block $I_{r'-k}$ is sharp at both its lower and upper limits.
\end{remark}

We recall the following rigidity theorem about orthogonal maps between projective spaces from \cite{GN}.

\begin{theorem}[\cite{GN}]\label{thm:ball}
	Let $l, l' \ge 2$ in $\mathbb{N}$, and let $f$ be a local orthogonal map from $\mathbb{P}^{1,l}$ to $\mathbb{P}^{1,l'}$.
	\begin{itemize}
		\item[(1)] If $l' < l$, then $f$ is a constant map.
		\item[(2)] If $l \le l' \le 2l-2$, then $f$ is a linear map or a constant map.
	\end{itemize}
\end{theorem}

While Theorem \ref{main_shilov} establishes a general structural reduction for arbitrary gap intervals, its analytical power becomes particularly evident when applied to the lowest gaps. Specifically, when $k=1$, the essential component $\phi$ is forced to map from $S_{1,s}$ into $S_{1, s'-r'+1}$, completely eliminating the higher-rank complexity of the target space. This nature geometric reduction bridges our mixed-rank setting directly to the classical rank-1 rigidity theory. By coupling our strict block-diagonalization with established rigidity result Theorem \ref{thm:ball} for projective spaces , we obtain a complete and explicit classification of the mappings in these initial gap intervals. 

As a direct consequence of this reduction, we obtain the following rigidity corollary:

\begin{corollary}\label{cor:linear_embedding}
	Let $s \ge 2$, and let $r', s' \in \mathbb{N}$ with $2 \le r' \le s'$. Let $\Sigma \subset S_{1,s}$ be a connected open piece of the Shilov boundary, and let $f: \Sigma \to S_{r',s'}$ be a smooth CR mapping. 
	\begin{itemize}
		\item[(1)] If $s' - r' < s - 1$, then $f$ is a constant map.
		
		\item[(2)] If $s - 1 \le s' - r' < 2s - 2$, then after composing with suitable automorphisms of $\Omega_{1,s}$ and $\Omega_{r',s'}$, $f$ is given by the standard linear embedding:
		$$ z \mapsto \begin{pmatrix}
			I_{r'-1} & 0 & 0 \\
			0 & z & 0 
		\end{pmatrix}. $$
	\end{itemize}
\end{corollary}

The following example demonstrates that the inequalities in Corollary \ref{cor:linear_embedding} governing the rigidity bounds are sharp.

\begin{example}
	The generalized Whitney map $f$ from $
	S_{1,s}$ to $S_{r',r'+2s-2}$ is given by
	$$z=[z_1,\cdots,z_s] \mapsto \begin{pmatrix}
		z_1 & \cdots & z_{s-1} & z_1z_s & z_2z_s & \cdots & z_{s}^2 & 0 & \cdots & 0 \\ 	
		0 & \cdots & 0 & 0 & 0 & \cdots & 0 & 1 & \cdots & 0 \\	
		\vdots & \ddots & \vdots & \vdots & \vdots & \ddots & \vdots & \vdots & \ddots & \vdots \\	
		0 & \cdots & 0 & 0 & 0 & \cdots & 0 & 0 & \cdots & 1
	\end{pmatrix}.$$
\end{example}

{\bf Acknowledgements:}
The author is deeply grateful to Prof. Sung-Yeon Kim for her interesting question, offering guidance, the invaluable discussions, constructive suggestions and detailed feedback throughout the development of this paper. She also wishes to thank Prof. Sui-Chung Ng's invaluable discussion to refine the paper. Additionally, thanks are due to Prof. Xiaojun Huang, Prof. Chenyang Xu and Prof. Dmitri Zaitsev for their insightful comments and valuable suggestions.
 
 \noindent{\bf Data availability statement.} Data sharing not applicable to this article as no datasets were generated or analysed during the current study.
 
 \noindent{\bf Ethical Statement}
 This work is purely theoretical mathematical research. No human participants, animal experiments, clinical data, or biological samples are involved in this manuscript. The author confirms that this manuscript is original research, has not been submitted simultaneously to other journals, and has approved the submission. All sources cited in this work are properly acknowledged. The author states that there is no conflict of interest to declare.

\end{document}